\documentclass[reqno]{amsart}
\usepackage{hyperref}

\begin{document}

\title[\hfilneg \hfil Uniqueness of  meromorphic function sharing three small functions CM with its $n-$ exact difference]
{Uniqueness of meromorphic function sharing three small functions CM with its $n-$ exact difference}

\author[XiaoHuang Huang \hfil \hfilneg]
{XiaoHuang Huang}

\address{XiaoHuang Huang: Corresponding author\newline
Department of Mathematics, Shenzhen University, Shenzhen 518055, China}
\email{1838394005@qq.com}

\begin{abstract}
In this paper, we study the uniqueness of the shift of meromorphic functions. We prove: Let $f$ be a  non-constant meromorphic function satisfying $\rho_{2}(f)<1$, let $\eta$ be a non-zero complex number,  and let $a,b,c\in\hat{S}(f)$  be three  distinct  small functions. If $f$ and $\Delta^{n}_{\eta}f$ share $a,b,c$ CM, then $f\equiv \Delta^{n}_{\eta}f$.
\end{abstract}

\maketitle
\numberwithin{equation}{section}
\newtheorem{theorem}{Theorem}[section]
\newtheorem{lemma}[theorem]{Lemma}
\newtheorem{remark}[theorem]{Remark}
\newtheorem{corollary}[theorem]{Corollary}
\newtheorem{example}[theorem]{Example}
\newtheorem{problem}[theorem]{Problem}
\allowdisplaybreaks

\section{Introduction and main results}
Throughout this paper, we assume that the reader have a knowledge of  the fundamental results and the standard notations of the Nevanlinna value distribution theory. See(\cite{h3,y1,y2}). In the following, a meromorphic function $f$ means meromorphic in the whole complex plane. Define
 $$\rho(f)=\varliminf_{r\rightarrow\infty}\frac{log^{+}T(r,f)}{logr},$$
 $$\rho_{2}(f)=\varlimsup_{r\rightarrow\infty}\frac{log^{+}log^{+}T(r,f)}{logr}$$
by the order  and the hyper-order  of $f$, respectively. When $\rho_{0}(f)<\infty$, we say $f$ is of finite order.

By $S(r,f)$, we denote any quantity satisfying $S(r, f) = o(T(r, f))$, as $r\to \infty $ outside of a possible exceptional set of finite logarithmic measure. A meromorphic function $a(z)$ satisfying $T(r,a)=S(r,f)$ is called a small function of $f$.  We denote $S(f)$ as the family of all small meromorphic functions of $f$ which includes the constants in $\mathbb{C}$. Moreover, we define $\hat{S}(f)=S(f)\cup\{\infty\}$. We say that two non-constant meromorphic functions $f$ and $g$ share small function $a$ CM(IM) if $f-a$ and $g-a$ have the same zeros counting multiplicities (ignoring multiplicities).

 We say that two nonconstant meromorphic functions  $f$ and $g$ share small function $a$ CM(IM) if $f-a$ and $g-a$ have the same zeros counting multiplicities (ignoring multiplicities).   And we that $f(z)$ and $g(z)$ share $a$ CM almost if
$$N(r,\frac{1}{f-a})+N(r,\frac{1}{g-a})-2N(r,f=a=g)=S(r,f)+S(r,g).$$

Let $f(z)$ be a meromorphic function, and a finite complex number $\eta$, we define  its difference operators by
\begin{equation*}
\Delta_\eta f(z)=f(z+\eta)-f(z), \quad \Delta_\eta^{n}f(z)=\Delta_{\eta}^{n-1}(\Delta_\eta f(z)).
\end{equation*}

In 1977, Rubel and Yang \cite{ruy}  considered the uniqueness of an entire function and its derivative. They proved.

\

{\bf Theorem A} \ Let $f$ be a non-constant entire function, and let $a, b$ be two finite distinct complex values. If $f$ and $f'$
 share $a, b$ CM, then $f\equiv f'$.

In recent years, there has been tremendous interests in developing  the value distribution of meromorphic functions with respect to difference analogue,  see [2-9, 11-15, 20]. Heittokangas et al \cite{hkl} proved a similar result analogue of Theorem A concerning shift.

In 2011,  Heittokangas-Korhonen- Laine-Rieppo \cite{hkl} improved Theorem B to meromorphic function. They proved.

\

{\bf Theorem B} Let $f$ be a  non-constant meromorphic function of finite order, and let $\eta$ be a non-zero complex number. If $f$ and $f_{\eta}$ share three distinct values $a,b,c$ CM in the extended plane, then $f\equiv f_{\eta}$.

In 2014, Halburd-Korhonen-Tohge \cite{h3} investigated the relationship of characteristic functions between $f$ and $f_{\eta}$ in $\rho_{2}(f)<1$.

In 2019, Deng-Fang-liu \cite{lyf} considered to replace sharing values by sharing small functions. They proved

\

{\bf Theorem C}  Let $f$ be a  non-constant entire function of finite order, let $\eta$ be a non-zero complex number, $n$ a positive integer, and let $a\not\equiv\infty,b\not\equiv\infty$  be two  distinct  small functions of $f$. If $f$ and $\Delta^{n}_{\eta}f$ share $a,b$ CM, then $f\equiv \Delta^{n}_{\eta}f$.

\

{\bf Remark 1} Theorem C holds when $f$ is a non-constant meromorphic function of $\rho_{2}(f)<1$ such that $N(r,f)=S(r,f)$.

Next, Heittokangas, et. \cite{hkl1} proved.

\

{\bf Theorem D} Let $f$ be a  non-constant meromorphic function of finite order, let $\eta$ be a nonzero finite value, and let $a\not\equiv\infty$, $b\not\equiv\infty$  and $d\not\equiv\infty\in \hat{S}(f)$ be three distinct small functions  such that $a$, $b$ and $c$ are periodic functions with period $\eta$. If $f$ and $f_{\eta}$ share $a,b,c$ CM, then $f(z)\equiv f(z+\eta)$.

In this paper, we improve  Theorem D  from finite order to $\rho_{2}(f)<1$, the first exact difference operator to $n-$exact difference operators, and three distinct periodic functions to any three distinct small functions. We prove

\

{\bf Theorem 1}  Let $f$ be a  non-constant meromorphic function satisfying $\rho_{2}(f)<1$, let $\eta$ be a non-zero complex number,  and let $a,b,c\in\hat{S}(f)$  be three  distinct  small functions. If $f$ and $\Delta^{n}_{\eta}f$ share $a,b,c$ CM, then $f\equiv \Delta^{n}_{\eta}f$.

Recently, the author\cite{h} obtained

\

{\bf Theorem E}
 Let $f$ be a  transcendental entire function of finite order, let $\eta\neq0$ be a finite complex number, $n\geq1, k\geq0$  two  integers and let $a, b$ be two  distinct finite complex values. If $f$ and $(\Delta_{\eta}^{n}f)^{(k)}$ share $a$ CM and share $b$ IM, then $a=2b$, and either $f(z)\equiv(\Delta_{\eta}^{n}f(z))^{(k)}$ or
 $$f=be^{2(c_{1}z+d)}-2be^{c_{1}z+d}+2b,$$
$$(\Delta_{\eta}^{n}f)^{(k)}=be^{c_{1}z+d},$$
where $c_{1}=(-2)^{-\frac{n+1}{k}}$ for $k\geq1$ and $d$ are two finite constants. Especially, when $k=0$, there is only the case $f\equiv\Delta_{\eta}^{n}f$ occurs.

So it is naturally to raise a conjecture that

\

{\bf Conjecture}  Let $f$ be a  non-constant meromorphic function satisfying $\rho_{2}(f)<1$, let $\eta$ be a non-zero complex number,  and let $a,b,c\in\hat{S}(f)$  be three  distinct  small functions. If $f$ and $\Delta^{n}_{\eta}f$ share $a,b$ CM and share $c$ IM, is $f\equiv \Delta^{n}_{\eta}f$?

\

{\bf Example 1} \ Let $f(z)=\frac{e^{z}}{1-e^{-2z}}$, and let $\eta=\pi i$. Then $\Delta^{n}_{\eta}f(z)=(-2)^{n}\frac{-e^{z}}{1-e^{-2z}}$, and $f(z)$ and $\Delta^{n}_{\eta}f(z)$ share $0,\infty$ CM, but $f(z)\not\equiv \Delta^{n}_{\eta}f(z)$.

\

{\bf Example 2} \ Let $f(z)=e^{z}$, and let $\eta=\pi i$. Then $\Delta^{n}_{\eta}f(z)=(-2)^{n}e^{z}$, and $f(z)$ and $\Delta^{n}_{\eta}f(z)$ share $0,\infty$ CM, $f(z)$ and $\Delta^{n}_{\eta}f(z)$ attain different values everywhere in the complex plane, but $f(z)\not\equiv \Delta^{n}_{\eta}f(z)$.

Above two examples of  show that "3CM" is necessary.

\section{Some Lemmas}
\begin{lemma}\label{21l}\cite{h3} Let $f$ be a non-constant meromorphic function of $\rho_{2}(f)<1$,  and let $\eta$ be a non-zero complex number. Then
$$m(r,\frac{f_{\eta}}{f})=S(r, f),$$
for all r outside of a possible exceptional set $E_{1}$ with finite logarithmic measure.
\end{lemma}

\begin{lemma}\label{23l}\cite{h3}  Let $f$ be a non-constant meromorphic function of $\rho_{2}(f)<1$,  and let $\eta$ be a non-zero complex number. Then
$$T(r,f)=T(r,f_{\eta})+S(r,f),$$
for all r outside of a possible exceptional set $E_{1}$ with finite logarithmic measure.
\end{lemma}

\begin{lemma}\label{28l}\cite{y1} Let $f(z)$  be a non-constant meromorphic function, and let $a_{j}\in \hat{S}(f)$ be $q$ distinct small functions for all $j=1,2,3$. Then
$$T(r,f)\leq \sum_{j=1}^{3}\overline{N}(r,\frac{1}{f-a_{j}})+S(r,f).$$
 \end{lemma}

\begin{lemma}\label{24l}\cite{hxh1} Let $f(z)$  be a  non-constant meromorphic function, and let $a(z)\not\equiv\infty, b(z)\not\equiv\infty$ be two distinct small meromorphic functions of $f(z)$. Suppose
\[L(f)=\left|\begin{array}{rrrr}a-b& &f-a \\
a'-b'& &f'-a'\end{array}\right|\]
then $L(f)\not\equiv0$.
\end{lemma}

\begin{lemma}\label{25l}\cite{hxh1}  Let $f(z)$  be a non-constant meromorphic function, and let $a(z)\not\equiv\infty, b(z)\not\equiv\infty$ be two distinct small meromorphic functions of $f(z)$. Then
$$m(r,\frac{L(f)}{f-a})=S(r,f), \quad m(r,\frac{L(f)}{f-b})=S(r,f).$$
And
$$m(r,\frac{L(f)f}{(f-a)(f-b)})=S(r,f),$$
where $L(f)$ is defined as in Lemma 2.4.
\end{lemma}

\begin{lemma}\label{27l}\cite{y1} Let $f$ and $g$ be two non-constant meromorphic functions. If $f$ and $g$ share $0,1,\infty$ CM, then
$$N_{(2}(r,f)+N_{(2}(r,\frac{1}{f})+N_{(2}(r,\frac{1}{f-1})=S(r,f).$$
\end{lemma}

\begin{lemma}\label{28l}\cite{y1} Let $f$ and $g$ be two non-constant meromorphic functions. If $f$ and $g$ share $0,1,\infty$ CM, and $f$ is not a M$\ddot{o}$bius transformation of $g$,  then\\
(i) $T(r,f)=N(r,\frac{1}{g'})+N_{0}(r)+S(r,f), T(r,g)=N(r,\frac{1}{f'})+N_{0}(r)+S(r,f)$, where $N_{0}(r)$ denotes the zeros of $f-g$, but not the zeros of $f$, $f-1$, and $\frac{1}{f}$.\\
(ii) $T(r,f)+T(r,g)=N(r,f)+N(r,\frac{1}{f})+N(r,\frac{1}{f-1})+N_{0}(r)+S(r,f)$;\\
(iii) $T(r,f)=N(r,\frac{1}{f-a})+S(r,f),$ where $a\neq0,1,\infty$.
\end{lemma}

\begin{lemma}\label{2010}\cite{y1} Let $f$ and $g$ be two non-constant meromorphic functions. If $f$ and $g$ share $0,1,\infty$ CM, and
$$N(r,f)\neq T(r,f)+S(r,f),$$
$$N(r,\frac{1}{f-a})\neq T(r,f)+S(r,f),$$
where $a\neq0,1,\infty$. Then $a,\infty$ are the Picard exceptional values of $f$, and $1-a,\infty $ are the Picard exceptional values of $g$.
\end{lemma}

\begin{lemma}\label{2011} \cite{a} Let $f$ and $g$ be two non-constant meromorphic functions. If $f$ and $g$ share $0,1,\infty$ CM, and $f$ is  a M$\ddot{o}$bius transformation of $g$,  then $f$ and $g$ assume one of the following six relations: (i) $fg=1$; (ii) $(f-1)(g-1)=1$; (iii) $f+g=1$; (iv) $f=cg$; (v) $f-1=c(g-1)$; (vi) $[(c-1)f+1][(c-1)g-c]=-c$, where $c\neq0,1$ is a complex number.
\end{lemma}

\begin{lemma}\label{23l}Let $f_{1}$ and $f_{2}$ be  non-constant meromorphic functions in $|z|<\infty$, then
$$N(r,f_{1}f_{2})-N(r,\frac{1}{f_{1}f_{2}})=N(r,f_{1})+N(r,f_{2})-N(r,\frac{1}{f_{1}})-N(r,\frac{1}{f_{2}}),$$
where $0<r<\infty$.
\end{lemma}

\begin{lemma}\label{291}\cite{y1}
Suppose $f_{1}, f_{2},\cdots, f_{n}(n\neq2)$   are meromorphic functions and $g_{1}, g_{2},\cdots, g_{n}$ are entire functions such that\\
(i) $\sum_{j=1}^{n}f_{j}e^{g_{j}}=0$,\\
(ii) $g_{j}-g_{k}$ are not constants for $1\leq j<k\leq n$,\\
(iii) For $1\leq j\leq n$ and $1\leq h<k\leq n$,
$$T(r,f_{j})=S(r,e^{g_{j}-g_{k}})(r\rightarrow\infty, r\not\in E_{2}).$$
Then $f_{j}\equiv0$ for all $1\leq j\leq n$.
\end{lemma}

\begin{lemma}\label{2010}\cite{g}
 Let $f$, $F$ and $g$ be three non-constant meromorphic functions, where $g=F(f)$. Then $f$ and $g$ share three values IM if and only if there exist an entire function $h$ such that,
 by a  suitable linear fractional transformation, one of the following cases holds: \\
 (i) $f\equiv g$;\\
 (ii) $f=e^{h}$ and $g=a(1+4ae^{-h}-4a^{2}e^{-2h})$ have three IM shared values $a\neq0$, $b=2a$ and $\infty$;\\
 (iii) $f=e^{h}$ and $g=\frac{1}{2}(e^{h}+a^{2}e^{-h})$ have three IM shared values $a\neq0$, $b=-a$ and $\infty$;\\
 (iv) $f=e^{h}$ and $g=a+b-abe^{-h}$ have three IM shared values $ab\neq0$ and $\infty$;\\
 (v) $f=e^{h}$ and $g=\frac{1}{b}e^{2h}-2e^{h}+2b$ have three IM shared values $b\neq0$, $a=2b$ and $\infty$;\\
 (vi) $f=e^{h}$ and $g=b^{2}e^{-h}$ have three IM shared values $a\neq0$, $0$ and $\infty$.
 \end{lemma}

\begin{lemma}\label{2010}\cite{ly4}[Lemma 7]
Let $f$ and $g$ be two non-constant meromorphic functions satisfying
$$\overline{N}(r,f)+\overline{N}(r,g)+\overline{N}(r,\frac{1}{f})+\overline{N}(r,\frac{1}{f})=S(r,f).$$
If $f^{s}g^{t}\equiv1$ for all integers $s$ and $t$($|s|+|t|>0$), then for any positive number $\varepsilon$, we have
$$N_{0}(r,1;f;g)\leq\varepsilon (T(r,f)+T(r,g))+S(r),$$
where  $N_{0}(r,1;f;g)$ denotes the reduced counting function of $f$ and $g$ related to the common $1$-points and $S(r)=o(T(r,f)+T(r,g))$ as $r\rightarrow\infty, r\not\in E_{3}$
 \end{lemma}

\

{\bf Remark 2} It is a tedious but "repeated" work in the proof of Lemma 2.8 that as for a small function $a\not\equiv0,1,\infty$, one can verify  $$2T(r,f)\leq3N(r,\frac{1}{f-a})+S(r,f),$$
holds when one of $N(r,\frac{1}{f})=S(r,f)$, $N(r,\frac{1}{f-1})=S(r,f)$ and $N(r,f)=S(r,f)$ occurs,
and
$$T(r,f)\leq2N(r,\frac{1}{f-a})+S(r,f),$$
 holds when neither $N(r,\frac{1}{f})=S(r,f)$ nor $N(r,\frac{1}{f-1})=S(r,f)$ nor $N(r,f)=S(r,f)$ occurs.

\

{\bf Remark 3} Denote $E$ by $E_{1}\bigcup E_{2}\bigcup E_{3}$, where $E_{1}$, $E_{2}$ and $E_{3}$ are defined as in Lemma 2.1, Lemma 2.2, Lemma 2.11 and Lemma 2.12. In the following proof, by $S(r,f)$, we denote any quantity satisfying $S(r, f) = o(T(r, f))$, as $r\to \infty $ outside of a possible exceptional set $E$.

\section{The proof of Theorem 1 }
  Suppose $f\not\equiv \Delta^{n}_{\eta}f$. Without lose of generality, we discuss two cases, i.e. $c\equiv\infty$ and $c\not\equiv\infty$.

{\bf Case 1} $c\equiv\infty$. Since $f$ is a non-constant meromorphic function satisfying $\rho_{2}(f)<1$, and $f$ and $\Delta^{n}_{\eta}f$ share $a,b,\infty$ CM, we know that there are two entire functions $p$ and $q$ such that
\begin{align}
\frac{\Delta^{n}_{\eta}f-a}{f-a}=e^{p}, \quad \frac{\Delta^{n}_{\eta}f-b}{f-b}=e^{q}.
\end{align}

Set
\begin{align}
\varphi=\frac{L(f)(f-\Delta^{n}_{\eta}f)}{(f-a)(f-b)},
\end{align}
where $L(f)\not\equiv0$ is defined as in Lemma 2.4. Since $f\not\equiv \Delta^{n}_{\eta}f$, then $\varphi\not\equiv0$.

 Set $F=\frac{f-a}{b-a}$ and $G=\frac{\Delta^{n}_{\eta}f-a}{b-a}$, and thus $F$ and $G$ share $0,1,\infty$ CM, as $f$ and $\Delta^{n}_{\eta}f$ share $a,b,\infty$ CM. Then by Lemma 2.6, we have
\begin{align}
N(r,f)=N_{1}(r,f),  N(r,\frac{1}{f-a})=N_{1}(r,\frac{1}{f-a}), N(r,\frac{1}{f-b})=N_{1}(r,\frac{1}{f-b}).
\end{align}
Since $f$ is a non-constant  meromorphic function satisfying $\rho_{2}(f)<1$, by Lemma 2.8, we have
\begin{align}
T(r,f)=T(r,F)+S(r,f)=T(r,G)+S(r,f)=T(r,g)+S(r,f).
\end{align}

We claim that
\begin{align}
T(r,f)=N(r,f)+S(r,f).
\end{align}
Otherwise, by Lemma 2.8, we know $N(r,f)=S(r,f)$, and hence {\bf Remark 1} implies $f\equiv \Delta^{n}_{\eta}f$, a contradiction. We also claim that $F$ is not a M$\ddot{o}$bius transformation of $G$. Otherwise, by Lemma 2.9, if  (i) occurs, we can see that
 \begin{align}
 N(r,\frac{1}{f-a})= N(r,\frac{1}{\Delta^{n}_{\eta}f-a})=S(r,f), N(r,f)= N(r,\Delta^{n}_{\eta}f)=S(r,f).
\end{align}
Then by Theorem C and {\bf Remark 1}, we can obtain a contradiction.\\

If (ii) occurs, we can see that
 \begin{align}
 N(r,\frac{1}{f-b})= N(r,\frac{1}{\Delta^{n}_{\eta}f-b})=S(r,f), N(r,f)= N(r,\Delta^{n}_{\eta}f)=S(r,f).
\end{align}
Then by  {\bf Remark 1}, we can obtain a contradiction.

If (iii) occurs, then by the definitions of $F$ and $G$, we can get $f+\Delta^{n}_{\eta}f=a+b$, that is
 \begin{align}
 &N(r,\frac{1}{f-a})= N(r,\frac{1}{\Delta^{n}_{\eta}f-a})=S(r,f),\notag\\
 &N(r,\frac{1}{f-b})= N(r,\frac{1}{\Delta^{n}_{\eta}f-b})=S(r,f).
\end{align}
Then by Lemma 2.1 and (3.8), we have
\begin{eqnarray*}
\begin{aligned}
 2T(r,f)&=m(r,\frac{1}{f-a})+m(r,\frac{1}{f-b})+S(r,f)\notag\\
 &=m(r,\frac{\Delta^{n}_{\eta}(f-a)}{f-a})+m(r,\frac{\Delta^{n}_{\eta}(f-b)}{f-b})+m(r,\frac{1}{\Delta^{n}_{\eta}(f-a)})\notag\\
 &+m(r,\frac{1}{\Delta^{n}_{\eta}(f-b)})+S(r,f)\notag\\
& \leq m(r,\frac{1}{\Delta^{n}_{\eta}(f-a)})+m(r,\frac{1}{\Delta^{n}_{\eta}(f-b)})+S(r,f)\leq 2T(r,\Delta^{n}_{\eta}f)\notag\\
 &-N(r,\frac{1}{\Delta^{n}_{\eta}f-\Delta^{n}_{\eta}a})-N(r,\frac{1}{\Delta^{n}_{\eta}f-\Delta^{n}_{\eta}b})+S(r,f),
\end{aligned}
\end{eqnarray*}
which implies
 \begin{align}
 N(r,\frac{1}{\Delta^{n}_{\eta}(f-a)})+N(r,\frac{1}{\Delta^{n}_{\eta}(f-b)})=S(r,f).
  \end{align}
Then we can know from Lemma 2.3, (3.8) and (3.9) that $\Delta^{n}_{\eta}a\equiv a$ or $\Delta^{n}_{\eta}a\equiv b$ and $\Delta^{n}_{\eta}b\equiv a$ or $\Delta^{n}_{\eta}b\equiv b$. If one of $\Delta^{n}_{\eta}a\equiv \Delta^{n}_{\eta}b\equiv a$ and $\Delta^{n}_{\eta}a\equiv \Delta^{n}_{\eta}b\equiv b$ occurs, then we have  $a\equiv b$, a contradiction. Hence $\Delta^{n}_{\eta}a\not\equiv \Delta^{n}_{\eta}b$. If $\Delta^{n}_{\eta}a\equiv a$ and $\Delta^{n}_{\eta}b\equiv b$, we set
\begin{align}
D&=(f-a)(\Delta^{n}_{\eta}a-\Delta^{n}_{\eta}b)-(\Delta^{n}_{\eta}f-\Delta^{n}_{\eta}a)(a-b)\notag\\
&=(f-b)(\Delta^{n}_{\eta}a-\Delta^{n}_{\eta}b)-(\Delta^{n}_{\eta}f-\Delta^{n}_{\eta}b)(a-b).
\end{align}
We claim that $D\not\equiv0$. Otherwise,  by the equalities $\Delta^{n}_{\eta}a\equiv a$ and $\Delta^{n}_{\eta}b\equiv b$, and the definition of $D$, we can get $f\equiv \Delta^{n}_{\eta}f$, a contradiction. So $D\not\equiv0$. Then it is easy to see that
\begin{align}
 2T(r,f)&=m(r,\frac{1}{f-a})+m(r,\frac{1}{f-b})+S(r,f)\notag\\
 &=m(r,\frac{1}{f-a}+\frac{1}{f-b})+S(r,f)\notag\\
  &=m(r,\frac{D}{f-a}+\frac{D}{f-b})+m(r,\frac{1}{D})\notag\\
 & \leq m(r,f-\Delta^{n}_{\eta}f)+N(r,f-\Delta^{n}_{\eta}f)+S(r,f)\notag\\
 &\leq m(r,f)+N(r,f)+S(r,f)=T(r,f)+S(r,f),
 \end{align}
which implies
\begin{align}
 T(r,f)=S(r,f),
\end{align}
a contradiction. Thus we know that it must occur that $\Delta^{n}_{\eta}a\equiv b$ and $\Delta^{n}_{\eta}b\equiv a$. And then $f-a\equiv -(\Delta^{n}_{\eta}f-\Delta^{n}_{\eta}a)$. Set $F_{1}=\frac{f-a}{b-a}$ and $G_{1}=\frac{\Delta^{n}_{\eta}f-\Delta^{n}_{\eta}a}{b-a}$, then $G_{1}=-F_{1}$. By (3.8) we know that $F_{1}$ and $G_{1}$ share $0,1,\infty$ CM. According to Lemma 2.12, we can know that either $F_{1}\equiv G_{1}$, combining $G_{1}=-F_{1}$, we get $F_{1}\equiv0$, that is $T(r,f)=T(r,F_{1})+S(r,f)=S(r,f)$, a contradiction. Or $N(r,f)=N(r,F_{1})+S(r,f)=S(r,f)$, and in this case we can use Theorem C and Remark 1 to obtain a contradiction.

If (iv) occurs, that is $F=jG$, where $j\neq0, 1$ is a finite constant. Then by Lemma 2.12, we can obtain either $F\equiv G$, or $N(r,F)=S(r,f)$. But we can obtain two contradictions from both of cases.

If (v) occurs, that is $F-1=i(G-1)$, where $i\neq0, 1$ is a finite constant. Then by Lemma 2.12, we can obtain either $F\equiv G$, or $N(r,F)=S(r,f)$. But we can obtain two contradictions from both of cases.

If (vi) occurs, $[(k-1)F+1][(k-1)G-k]=-k$, where $k\neq0,1$ is a complex number. We can see that
 \begin{align}
 N(r,F)+S(r,f)=N(r,f)= N(r,g)=N(r,F)+S(r,f)=S(r,f).
\end{align}
Then by  {\bf Remark 1} and Theorem C, we can obtain a contradiction.\\

Hence, $F$ is not a M$\ddot{o}$bius transformation of $G$. If $ab\equiv0$, and without lose of generality, we set $a\equiv0$.
Easy to see from (3.2), Lemma 2.1 and Lemma 2.5 that
\begin{eqnarray*}
\begin{aligned}
T(r,\varphi)&=m(r,\frac{L(f)(f-\Delta^{n}_{\eta}f)}{(f-a)(f-b)})+N(r,\varphi)\notag\\
&\leq m(r,\frac{L(f)f}{(f-a)(f-b)})+m(r,1-\frac{\Delta^{n}_{\eta}f}{f})+N(r,\varphi)\notag\\
&\leq N_{1}(r,f)+S(r,f),
\end{aligned}
\end{eqnarray*}
that is
\begin{align}
T(r,\varphi)\leq N_{1}(r,f)+S(r,f).
\end{align}

We also obtain
\begin{align}
m(r,\frac{\varphi}{f})&\leq m(r,\frac{L(f)}{(f-a)(f-b)})+m(r,1-\frac{\Delta^{n}_{\eta}f}{f})=S(r,f).
\end{align}
Then it follows from Lemma 2.7 and (3.14)-(3.15) that
\begin{eqnarray*}
\begin{aligned}
m(r,\frac{1}{f})&\leq m(r,\frac{\varphi}{f})+m(r,\frac{1}{\varphi})\\
&\leq T(r,\varphi)-N(r,\frac{1}{\varphi})+S(r,f)\\
&\leq T(r,\varphi)-(N(r,\frac{1}{L(f)})+N_{0}(r,\frac{1}{f-\Delta^{n}_{\eta}f}))+S(r,f)\\
&\leq N_{1}(r,f)-T(r,f)+S(r,f)=S(r,f),
\end{aligned}
\end{eqnarray*}
which is
\begin{align}
m(r,\frac{1}{f})=S(r,f).
\end{align}
Here, $N_{0}(r,\frac{1}{f-\Delta^{n}_{\eta}f})=N_{0}(r,\frac{1}{F-G})+S(r,f)$, $N_{0}(r,\frac{1}{F-G})$ denotes the zeros of $F-G$, but not the zeros of $F$, $F-1$, and $\frac{1}{F}$, and $N_{0}(r,\frac{1}{f-\Delta^{n}_{\eta}f})$ denotes the zeros of $f-\Delta^{n}_{\eta}f$, but not the zeros of $f$, $f-a$, and $f-b$. So
\begin{align}
T(r,f)=N(r,\frac{1}{f})+S(r,f).
\end{align}
Combing Lemma 2.7 and (3.17), we can get
\begin{eqnarray*}
\begin{aligned}
&N(r,f)+N(r,\frac{1}{f})+N(r,\frac{1}{f-b})+N_{0}(r)\\
&=T(r,f)+T(r,\Delta^{n}_{\eta}f)+S(r,f)\\
&=N(r,\frac{1}{f})+N(r,f)+S(r,f),
\end{aligned}
\end{eqnarray*}
that is
\begin{align}
N(r,\frac{1}{f-b})+N_{0}(r)=S(r,f),
\end{align}
and therefore by (3.18), we have
\begin{align}
T(r,e^{p})&=N(r,\frac{1}{e^{p}-1})+S(r,f)\notag\\
&\leq N_{0}(r)+N(r,\frac{1}{f-b})=S(r,f)
\end{align}
and
\begin{align}
&T(r,f)=m(r,\frac{1}{f-b})+N(r,\frac{1}{f-b})+S(r,f)\notag\\
&=m(r,\frac{1}{f-b})+S(r,f)\leq m(r,\frac{1}{\Delta^{n}_{\eta}f-\Delta^{n}_{\eta}b})+S(r,f)\notag\\
&\leq T(r,\Delta^{n}_{\eta}f)-N(r,\frac{1}{\Delta^{n}_{\eta}f-\Delta^{n}_{\eta}b})+S(r,f),
\end{align}
which implies
\begin{align}
N(r,\frac{1}{\Delta^{n}_{\eta}f-\Delta^{n}_{\eta}b})=S(r,f).
\end{align}
If $\Delta^{n}_{\eta}b=0$, then (3.17) deduces $T(r,f)=S(r,f)$, a contradiction. Hence $\Delta^{n}_{\eta}b=b$. Then by (3.1) and Lemma 2.1, we have
\begin{align}
m(r,e^{q})=m(r,\frac{\Delta^{n}_{\eta}f-\Delta^{n}_{\eta}b}{f-b})=S(r,f).
\end{align}
Solving the equation (3.1), we can get
\begin{align}
f=\frac{a-b+be^{q}-ae^{p}}{e^{q}-e^{p}}, \quad \Delta^{n}_{\eta}f=\frac{b-a+ae^{-p}-be^{-q}}{e^{-p}-e^{-q}}.
\end{align}
It follows from above, (3.19) and (3.22),  we have $T(r,f)=S(r,f)$, a contradiction. Thus, we know that neither $\Delta_{\eta}^{n}b=0$ nor $\Delta_{\eta}^{n}b=b$ holds. Then by Lemma 2.1-Lemma 2.3 and (3.21), we have
\begin{align}
N(r,\frac{1}{f-b})&=N(r,\frac{1}{\Delta^{n}_{\eta}f-b})+S(r,f)\notag\\
&=N(r,\frac{1}{\Delta_{\eta}^{n}f-\Delta_{\eta}^{n}b})+S(r,f)=S(r,f),
\end{align}
and
\begin{align}
T(r,f)=N(r,\frac{1}{f})+S(r,f)=N(r,f)+S(r,f).
\end{align}

Set
$$P_{1}=\frac{\Delta^{n}_{\eta}f-b}{f-b},\quad Q_{1}=\frac{(f-b)\Delta^{n}_{\eta}b}{(\Delta^{n}_{\eta}f-\Delta^{n}_{\eta}b)b}.$$
We ca see from above that
$$N(r,P_{1})+N(r,Q_{1})+N(r,\frac{1}{P_{1}})+N(r,\frac{1}{Q_{1}})=S(r,f).$$
If for all integers $s$ and $t$ satisfying ($|s|+|t|>0$) such that $P_{1}^{s}Q_{1}^{t}\not\equiv1$, then by Lemma 2.12, we get
\begin{align}
T(r,f)&=N(r,\frac{1}{f-a})+S(r,f)\leq N_{0}(r,P_{1};1;Q_{1})+S(r,f)\notag\\
&\leq\varepsilon(T(r,P_{1})+T(r,Q_{1}))+S(r,f)\notag\\
&\leq2\varepsilon T(r,f)+S(r,f),
\end{align}
it follows from above and $\varepsilon<\frac{1}{2}$ that $T(r,f)=S(r,f)$, a contradiction. Therefore, there exist two integer $s=1$ and $t=1$ such that $P_{1}Q_{1}\equiv1$. That is
$$\frac{(\Delta^{n}_{\eta}f-b)\Delta^{n}_{\eta}b}{(\Delta^{n}_{\eta}f-\Delta^{n}_{\eta}b)b}\equiv1.$$
Rewrite above as
\begin{align}
\frac{\Delta^{n}_{\eta}b-b}{\Delta^{n}_{\eta}f-\Delta^{n}_{\eta}b}\equiv \frac{b}{\Delta^{n}_{\eta}b}-1,
\end{align}
which follows from Lemma 2.1, Lemma 2.2 and (3.21) that
\begin{align}
T(r,f)=m(r,\frac{1}{f-b})+S(r,f)\leq m(r,\frac{1}{\Delta^{n}_{\eta}f-\Delta^{n}_{\eta}b})+S(r,f)=S(r,f),
\end{align}
a contradiction.

So $ab\not\equiv0$. By Lemma 2.7 and (3.5) that
\begin{eqnarray*}
\begin{aligned}
3T(r,f)+N_{0}(r)=2T(r,f)+m(r,\frac{1}{f-a})+m(r,\frac{1}{f-b})+S(r,f),
\end{aligned}
\end{eqnarray*}
which follows from above inequality that
\begin{align}
&T(r,f)+N_{0}(r)=m(r,\frac{1}{f-a})+m(r,\frac{1}{f-b})+S(r,f)\notag\\
&\leq m(r,\frac{\Delta^{n}_{\eta}f-\Delta^{n}_{\eta}a}{f-a})+m(r,\frac{\Delta^{n}_{\eta}f-\Delta^{n}_{\eta}b}{f-b})
+m(r,\frac{1}{\Delta^{n}_{\eta}f-\Delta^{n}_{\eta}a})+m(r,\frac{1}{\Delta^{n}_{\eta}f-\Delta^{n}_{\eta}b})\notag\\
&+S(r,f)\leq m(r,\frac{1}{\Delta^{n}_{\eta}f-\Delta^{n}_{\eta}a})+m(r,\frac{1}{\Delta^{n}_{\eta}f-\Delta^{n}_{\eta}b})+S(r,f).
\end{align}

We discuss two cases.

{\bf Case 1.1} $\Delta^{n}_{\eta}a\not\equiv\Delta^{n}_{\eta}b$.

{\bf Case 1.1.1} $\Delta^{n}_{\eta}a\not\equiv a,b$ and $\Delta^{n}_{\eta}b\not\equiv a,b$. Let $F_{1}=\frac{1}{F}$ and $G_{1}=\frac{1}{G}$. We only need to discuss $F_{1}$ is not a M$\ddot{o}$bius transformation of $G_{1}$.  We discuss two subcases.

{\bf Subcases 1.1.1.1} \quad $T(r,F_{1})\neq N(r,F_{1})+S(r,f)=N(r,\frac{1}{f-a})+S(r,f)$.  Then by Lemma 2.8, we have
\begin{align}
N(r,\frac{1}{f-a})=N(r,\frac{1}{\Delta^{n}_{\eta}f-a})=S(r,f).
\end{align}
Lemma 2.1 implies that
\begin{align}
N(r,\frac{1}{\Delta^{n}_{\eta}f-\Delta^{n}_{\eta}a})=S(r,f).
\end{align}
If $a\equiv \Delta^{n}_{\eta}a$, then by Lemma 2.1 and (3.1), we get
\begin{align}
m(r,e^{p})=m(r,\frac{\Delta^{n}_{\eta}f-\Delta^{n}_{\eta}a}{f-a})=S(r,f),
\end{align}
and then
\begin{align}
N(r,\frac{1}{f-b})+N_{0}(r)\leq N(r,\frac{1}{e^{p}-1})+S(r,f)=S(r,f).
\end{align}
(3.31) and (3.33) deduce
\begin{align}
m(r,e^{q})&=N(r,\frac{1}{e^{q}-1})+S(r,f)\leq N(r,\frac{1}{f-a})+N_{0}(r)\notag\\
&+S(r,f)=S(r,f).
\end{align}
Combining (3.23), (3.32) and (3.34), we have $T(r,f)=S(r,f)$, a contradiction.

Set
$$P_{2}=\frac{\Delta^{n}_{\eta}f-a}{f-a},\quad Q_{2}=\frac{(f-a)(b-\Delta^{n}_{\eta}a)}{(\Delta^{n}_{\eta}f-\Delta^{n}_{\eta}a)(b-a)}.$$
We ca see from above that
$$N(r,P_{2})+N(r,Q_{2})+N(r,\frac{1}{P_{2}})+N(r,\frac{1}{Q_{2}})=S(r,f).$$
With a similar method, we can obtain $T(r,f)=S(r,f)$, a contradiction.

{\bf Subcases 1.1.1.2} \quad $T(r,F_{1})= N(r,F_{1})+S(r,f)=N(r,\frac{1}{f-a})+S(r,f)$. It follows from (3.29) that
\begin{align}
N(r,\frac{1}{f-b})+N_{0}(r)=S(r,f).
\end{align}
As we set
$$P_{3}=\frac{\Delta^{n}_{\eta}f-b}{f-b},\quad Q_{3}=\frac{(f-b)(a-\Delta^{n}_{\eta}b)}{(\Delta^{n}_{\eta}f-\Delta^{n}_{\eta}b)(a-b)},$$
 and we can also obtain $T(r,f)=S(r,f)$, a contradiction.

{\bf Case 1.1.2} $\Delta^{n}_{\eta}a\equiv a$ and $\Delta^{n}_{\eta}b\equiv b$. Then by Lemma 2.1 and (3.23), we can get $T(r,f)=T(r,F)+S(r,f)=S(r,f)$, a contradiction.

{\bf Case 1.1.3} $\Delta^{n}_{\eta}a\equiv b$ and $\Delta^{n}_{\eta}b\equiv a$.  Then (3.23) implies
\begin{align}
f-a=\frac{d(t-e^{p})}{e^{2p}-t},\quad \Delta^{n}_{\eta}f-b=\frac{td(e^{p}-1)}{e^{2p}-t},
\end{align}
where $d=a-b$ and $t=e^{p+q}$. And then
\begin{align}
\Delta^{n}_{\eta}(f-a)=\Delta^{n}_{\eta}(\frac{d(t-e^{p})}{e^{2p}-t})=\frac{td(e^{p}-1)}{e^{2p}-t},
\end{align}
that is
\begin{align}
\sum_{i=1}^{2n+1}\sum_{l=0}^{n}A_{l,i}e^{ip_{l\eta}} + Be^{Q}\equiv0,
\end{align}
where $e^{Q}\equiv1$ is a constant and $A_{l,i}, B$ are small functions of $f$ for $l=0,1,2,\ldots,n$ and $i=1,2,3,\ldots,2n+1$. (3.1) and Lemma 2.1 deduce
\begin{align}
T(r,e^{p+q})=m(r,e^{p+q})=m(r,\frac{(\Delta^{n}_{\eta}f-\Delta^{n}_{\eta}b)(\Delta^{n}_{\eta}f-\Delta^{n}_{\eta}a)}{(f-b)(f-a)})=S(r,f),
\end{align}
\begin{align}
m(r,e^{p})=m(r,\frac{1}{f-a})+S(r,f),
\end{align}
and
\begin{align}
m(r,e^{q})=m(r,\frac{1}{f-b})+S(r,f).
\end{align}
Immediately,
\begin{align}
m(r,e^{p})=m(r,e^{q})+S(r,f).
\end{align}
We claim that for all $l=0,1,2,\ldots,n$ and $i=1,2,3,\ldots,2n+1$, $ip_{l\eta}-Q$ are not constants. Otherwise, if there exists a pair $(l,i)$ such that $ip_{l\eta}-Q$ is a constant, we can obtain that $m(r,e^{ip_{l\eta}})=im(r,e^{p})+S(r,f)=S(r,f)$. And then we will get from (3.23) and (3.42) that $T(r,f)=S(r,f)$, a contradiction. On the other hand, we have $T(r,A_{l,i})=S(r,f)$ and $T(r,B)=S(r,f)$ for $l=0,1,2,\ldots,n$ and $i=1,2,3,\ldots,2n+1$. Then by Lemma 2.11, we can get $A_{l,i}\equiv0$ for $l=0,1,2,\ldots,n$ and $i=1,2,3,\ldots,2n+1$ and $B\equiv 0$, that is $a\equiv b$, a contradiction.

{\bf Case 1.2} $\Delta^{n}_{\eta}a\equiv\Delta^{n}_{\eta}b$. Then by (3.24) we have
\begin{align}
T(r,f)+N_{0}(r)&=m(r,\frac{1}{f-a})+m(r,\frac{1}{f-b})+S(r,f)\notag\\
&\leq m(r,\frac{1}{f-a}+\frac{1}{f-b})+S(r,f)\notag\\
&\leq m(r,\frac{\Delta^{n}_{\eta}f-\Delta^{n}_{\eta}a}{f-a}+\frac{\Delta^{n}_{\eta}f-\Delta^{n}_{\eta}b}{f-b})
+m(r,\frac{1}{\Delta^{n}_{\eta}f-\Delta^{n}_{\eta}a})\notag\\
&+S(r,f)\leq m(r,\frac{1}{\Delta^{n}_{\eta}f-\Delta^{n}_{\eta}a})+S(r,f)\notag\\
&\leq T(r,\Delta^{n}_{\eta}f)-N(r,\frac{1}{\Delta^{n}_{\eta}f-\Delta^{n}_{\eta}a})+S(r,f),
\end{align}
it deduces that
\begin{align}
N(r,\frac{1}{\Delta^{n}_{\eta}f-\Delta^{n}_{\eta}a})+N_{0}(r)=S(r,f).
\end{align}
It follows from Lemma 2.7 that $\Delta^{n}_{\eta}a\equiv a$ or $\Delta^{n}_{\eta}a\equiv b$.  If  $\Delta^{n}_{\eta}a\equiv a$, then by Lemma 2.1 and (3.1) that
\begin{align}
T(r,e^{p})=m(r,e^{p})=m(r,\frac{\Delta^{n}_{\eta}f-\Delta^{n}_{\eta}a}{f-a})=S(r,f).
\end{align}
On the other hand, by Nevanlinna's Second Fundamental Theorem and (3.33), we have
\begin{align}
T(r,e^{q})&\leq N(r,\frac{1}{e^{q}-1})+S(r,f)\notag\\
&\leq N(r,\frac{1}{f-a})+N_{0}(r)=S(r,f).
\end{align}
By (3.23), (3.45) and (3.46) that  that $T(r,f)=S(r,f)$, a contradiction.\\

If  $\Delta^{n}_{\eta}b\equiv \Delta^{n}_{\eta}a\equiv b$, then using a similar proof of above, we can also obtain a contradiction.

{\bf Case 2}  $c\not\equiv\infty$.   Since $f$ and $\Delta^{n}_{\eta}f$ share $a,b,c$ CM, we set
$$F_{2}=\frac{f-a}{f-b}\cdot\frac{c-b}{c-a}, G_{2}=\frac{\Delta^{n}_{\eta}f-a}{\Delta^{n}_{\eta}f-b}\cdot\frac{c-b}{c-a},$$
and
\begin{align}
\frac{(f-b)(\Delta^{n}_{\eta}f-a)}{(f-a)(\Delta^{n}_{\eta}f-b)}=e^{h_{1}},  \frac{(f-c)(f_{\eta}-a)}{(f-a)(\Delta^{n}_{\eta}f-c)}=e^{h_{2}}, \frac{(f-b)(\Delta^{n}_{\eta}f-c)}{(f-c)(\Delta^{n}_{\eta}f-b)}=e^{h_{3}}.
\end{align}
And we know that $F_{2}$ and $G_{2}$ share $0,1,\infty$ CM. And we also have
\begin{align}
 N(r,\frac{1}{f-a})&=N_{1}(r,\frac{1}{f-a}),\notag\\
 N(r,\frac{1}{f-b})&=N_{1}(r,\frac{1}{f-b}),\notag\\
 N(r,\frac{1}{f-c})&=N_{1}(r,\frac{1}{f-c}).
\end{align}
We  claim that $F_{2}$ is not a M$\ddot{o}$bius transformation of $G_{2}$. Otherwise, then by Lemma 2.9, if (i) occurs, we can see that
 \begin{align}
 N(r,\frac{1}{f-a})= N(r,\frac{1}{\Delta^{n}_{\eta}f-a})=S(r,f), N(r,\frac{1}{f-b})= N(r,\frac{1}{\Delta^{n}_{\eta}f-b})=S(r,f).
\end{align}
By Lemma 2.11, we know that either $F_{2}\equiv G_{2}$, i.e $f\equiv \Delta^{n}_{\eta}f$. Or there is a non-constant entire function $h$ with $\rho(h)<1$ such that $F_{2}=e^{h}$ and $G_{2}=e^{-h}$. (3.47) implies $e^{h_{1}}\equiv \frac{1}{G_{2}^{2}}$, which is
\begin{align}
 T(r,e^{h_{1}})=2T(r,G_{3})+S(r,f)=2T(r,\Delta^{n}_{\eta}f)+S(r,f)=2T(r,f)+S(r,f).
\end{align}
Set
$$P_{4}=\frac{(f-a)(c-b)}{(f-b)(c-a)},\quad Q_{4}=\frac{(\Delta^{n}_{\eta}f-b)(c-a)}{(\Delta^{n}_{\eta}f-a)(c-b)}.$$
We can see from above that
$$N(r,P_{4})+N(r,Q_{4})+N(r,\frac{1}{P_{4}})+N(r,\frac{1}{Q_{4}})=S(r,f).$$
If for all integers $s$ and $t$ satisfying ($|s|+|t|>0$) such that $P_{4}^{s}Q_{4}^{t}\not\equiv1$, then by Lemma 2.12, we get
\begin{align}
T(r,f)&=N(r,\frac{1}{f-c})+S(r,f)\leq N_{0}(r,P_{4};1;Q_{4})+S(r,f)\notag\\
&\leq\varepsilon(T(r,P_{4})+T(r,Q_{4}))+S(r,f)\notag\\
&\leq2\varepsilon T(r,f)+S(r,f),
\end{align}
it follows from above and $\varepsilon<\frac{1}{2}$ that $T(r,f)=S(r,f)$, a contradiction. Therefore, there only exist two integer $s=1$ and $t=1$ such that $P_{4}Q_{4}\equiv1$. That is
$$e^{h_{1}}\equiv\frac{(f-a)(\Delta^{n}_{\eta}f-b)}{(f-b)(\Delta^{n}_{\eta}f-a)}\equiv1.$$
It follows from (3.50) that $T(r,f)=S(r,f)$, a contradiction.

If (ii) occurs, we can see that
 \begin{align}
 N(r,\frac{1}{f-b})= N(r,\frac{1}{\Delta^{n}_{\eta}f-b})=S(r,f), N(r,\frac{1}{f-c})= N(r,\frac{1}{\Delta^{n}_{\eta}f-c})=S(r,f).
\end{align}
And similar to the proof of (i), we can obtain a contradiction.

If (iii) occurs, we can see that
 \begin{align}
 N(r,\frac{1}{f-a})= N(r,\frac{1}{\Delta^{n}_{\eta}f-a})=S(r,f), N(r,\frac{1}{f-c})= N(r,\frac{1}{\Delta^{n}_{\eta}f-c})=S(r,f).
\end{align}
And similar to the proof of (i), we can obtain a contradiction.

If (iv) occurs, that is $F_{2}=dG_{2}$, where $d\neq0, 1$ is a finite constant. Then by Lemma 2.12, we can obtain either $F_{2}\equiv G_{2}$, or $N(r,F_{2})=S(r,f)$. But we can obtain two contradictions from both of cases.

If (v) occurs, we can obtain a contradiction with a similar  proof of (vi).

If (vi) occurs, we can see that
 \begin{align}
 N(r,\frac{1}{f-b})= N(r,\frac{1}{\Delta^{n}_{\eta}f-b})=S(r,f), N(r,\frac{1}{f-\frac{a(d-1)q+b}{(d-1)q-1}})=S(r,f),
\end{align}
where $q=\frac{c-b}{c-a}$, and $b\not\equiv \frac{a(d-1)q+b}{(d-1)q-1}$. And we can obtain a contradiction with a similar  proof of (i).

Therefore, $F_{2}$ is not a M$\ddot{o}$bius transformation of $G_{2}$.  We discuss two subcases.

 {\bf Subcase 2.1} $T(r,f)\neq N(r,\frac{1}{f-b})+S(r,f)$.  Then by Lemma 2.8 we know that
 \begin{align}
N(r,\frac{1}{f-b})=S(r,f).
\end{align}
Then  by (ii) of Lemma 2.7 and above, we have
\begin{eqnarray*}
\begin{aligned}
3T(r,f)+N_{0}(r)&=2T(r,f)+T(r,\Delta^{n}_{\eta}f)+m(r,\frac{1}{f-a})+m(r,\frac{1}{f-c})+S(r,f),
\end{aligned}
\end{eqnarray*}
which implies
 \begin{align}
T(r,f)+N_{0}(r)=T(r,\Delta^{n}_{\eta}f)+m(r,\frac{1}{f-a})+m(r,\frac{1}{f-c})+S(r,f).
\end{align}
Furthermore, by Lemma 2.1, we know
\begin{eqnarray*}
\begin{aligned}
T(r,f)&=m(r,\frac{1}{f-b})+S(r,f)\leq m(r,\frac{1}{\Delta^{n}_{\eta}f-\Delta^{n}_{\eta}b})+S(r,f)\\
&\leq T(r,\Delta^{n}_{\eta}f)-N(r,\frac{1}{\Delta^{n}_{\eta}f-\Delta^{n}_{\eta}b})+S(r,f),
\end{aligned}
\end{eqnarray*}
it follows from Lemma 2.8 that
\begin{align}
T(r,f)+N(r,\frac{1}{\Delta^{n}_{\eta}f-\Delta^{n}_{\eta}b})\leq T(r,g)+S(r,f).
\end{align}
If $\Delta^{n}_{\eta}b\equiv b$, then Lemma 2.1, Lemma 2.11 and (3.47) deduce
\begin{eqnarray*}
\begin{aligned}
T(r,e^{h_{1}})&=m(r,e^{h_{1}})=m(r,e^{-h_{1}})=m(r,\frac{(f-a)(\Delta^{n}_{\eta}f-b)}{(f-b)(\Delta^{n}_{\eta}f-a)})\\
&\leq m(r,\frac{\Delta^{n}_{\eta}f-b}{f-b})+m(r,\frac{f-a}{\Delta^{n}_{\eta}f-a})\\
&\leq m(r,\frac{1}{f-a})+N(r,\frac{\Delta^{n}_{\eta}f-a}{f-a})-N(r,\frac{f-a}{\Delta^{n}_{\eta}f-a})\\
&\leq  m(r,\frac{1}{f-a})+N(r,\Delta^{n}_{\eta}f)+N(r,\frac{1}{f-a})-N(r,f)\\
&-N(r,\frac{1}{\Delta^{n}_{\eta}f-a})+S(r,f)\leq m(r,\frac{1}{f-a})+S(r,f),
\end{aligned}
\end{eqnarray*}
which implies
\begin{align}
T(r,e^{h_{1}})+N(r,f)\leq m(r,\frac{1}{f-a})+N(r,\Delta^{n}_{\eta}f)+S(r,f).
\end{align}
We also have
\begin{align}
T(r,e^{h_{3}})+N(r,f)\leq m(r,\frac{1}{f-c})+N(r,\Delta^{n}_{\eta}f)+S(r,f).
\end{align}

Applying Lemma 2.3 to $e^{h_{1}}$ and $e^{h_{3}}$, we have
\begin{align}
T(r,e^{h_{1}})=\overline{N}(r,\frac{1}{e^{h_{1}}-1})+S(r,f),
\end{align}
and
\begin{align}
T(r,e^{h_{3}})=\overline{N}(r,\frac{1}{e^{h_{3}}-1})+S(r,f).
\end{align}

It follows from (3.55)-(3.56) and  (3.58)-(3.61) that
\begin{eqnarray*}
\begin{aligned}
N(r,\frac{1}{f-c})+N_{0}(r)+N(r,f)&=\overline{N}(r,\frac{1}{e^{h_{1}}-1})+N(r,f)+S(r,f)\\
&\leq m(r,\frac{1}{f-a})+N(r,\Delta^{n}_{\eta}f)+S(r,f),
\end{aligned}
\end{eqnarray*}
which is
\begin{align}
N(r,\frac{1}{f-c})+N_{0}(r)+N(r,f)\leq m(r,\frac{1}{f-a})+N(r,\Delta^{n}_{\eta}f)+S(r,f).
\end{align}
Similarly, we have
\begin{align}
N(r,\frac{1}{f-a})+N_{0}(r)+N(r,f)\leq m(r,\frac{1}{f-c})+N(r,\Delta^{n}_{\eta}f)+S(r,f).
\end{align}
From the fact that the zero of $f-g$ which are not $f-a$, $f-b$ nor $f-c$ are the zeros of $e^{h_{1}}-1$, that is
$N_{0}(\frac{1}{f-\Delta^{n}_{\eta}f})\leq N(\frac{1}{e^{h_{1}}-1})$. And hence by (3.62) and (3.63), we have
\begin{eqnarray*}
\begin{aligned}
&2N_{0}(r)+N(r,\frac{1}{f-a})+N(r,\frac{1}{f-c})+T(r,f)\\
&\leq m(r,\frac{1}{f-a})+m(r,\frac{1}{f-c})+2N(r,\Delta^{n}_{\eta}f)+S(r,f),
\end{aligned}
\end{eqnarray*}
which follows from (3.56) that
\begin{align}
N_{0}(r)+N(r,\frac{1}{f-a})+N(r,\frac{1}{f-c})\leq T(r,\Delta^{n}_{\eta}f)+S(r,f).
\end{align}
Then by (ii) of Lemma 2.7,  (3.64) and above, we can obtain $T(r,f)=S(r,f)$, a contradiction. Therefore $\Delta^{n}_{\eta}b\equiv a$ or $\Delta^{n}_{\eta}b\equiv c$. If $\Delta^{n}_{\eta}b\equiv a$, we set $F_{3}=\frac{1}{F_{2}}$ and $G_{3}=\frac{1}{G_{2}}$. Similar to the proof of above, we discuss the case that $F_{3}$ is not a M$\ddot{o}$bius transformation of $G_{3}$.

{\bf Subcase 2.1.1} $T(r,f)\neq N(r,\frac{1}{f-a})+S(r,f)$. Then by Lemma 2.8 we know that
 \begin{align}
N(r,\frac{1}{f-a})=S(r,f).
\end{align}
Then similar to the proof of the {\bf Case 2-(i)}, we can obtain a contradiction.

{\bf Subcase 2.1.2} $T(r,f)= N(r,\frac{1}{f-a})+S(r,f)$. Then by Lemma 2.8 we know that
 \begin{align}
T(r,f)=T(r,\Delta^{n}_{\eta}f)+S(r,f).
\end{align}
Then by (3.57), (3.66) and the fact that $T(r,f)= N(r,\frac{1}{f-a})+S(r,f)$, we obtain $T(r,f)=S(r,f)$, a contradiction.

If $\Delta^{n}_{\eta}b\equiv c$, we can also get a contradiction. Thus, neither $\Delta^{n}_{\eta}b\not\equiv a$, nor $\Delta^{n}_{\eta}b\not\equiv b$, nor $\Delta^{n}_{\eta}b\not\equiv c$. Then by {\bf Remark 2} and (3.57), we can obtain
 \begin{align}
3T(r,f)\leq T(r,\Delta^{n}_{\eta}f)+S(r,f).
\end{align}
From the fact that $f$ and $g$ share $a,b,c$ CM, and $N(r,\frac{1}{f-b})=S(r,f)$, we can get from the second fundamental theorem of Nevanlinna that
\begin{align}
T(r,\Delta^{n}_{\eta}f)&\leq N(r,\frac{1}{\Delta^{n}_{\eta}f-a})+N(r,\frac{1}{\Delta^{n}_{\eta}f-b})+N(r,\frac{1}{\Delta^{n}_{\eta}f-c})+S(r,f)\notag\\
&=N(r,\frac{1}{f-a})+N(r,\frac{1}{f-c})+S(r,f)\leq2T(r,f)+S(r,f),
\end{align}
which follows from (3.67) that $T(r,f)=S(r,f)$, a contradiction.

 {\bf Subcase 2.2} $T(r,f)= N(r,\frac{1}{f-b})+S(r,f)$.  Then by Lemma 2.7 and Lemma 2.8 we know that
 \begin{eqnarray*}
\begin{aligned}
2T(r,f)&=N(r,\frac{1}{f-a})+N(r,\frac{1}{f-b})+N(r,\frac{1}{f-c})+N_{0}(r)+S(r,f)\\
&=T(r,f)+N(r,\frac{1}{f-a})+N(r,\frac{1}{f-c})+N_{0}(r)+S(r,f)\\
&=N(r,\frac{1}{f-\Delta^{n}_{\eta}f})+S(r,f),
\end{aligned}
\end{eqnarray*}
it follows that
\begin{align}
&T(r,f)=N(r,\frac{1}{f-a})+N(r,\frac{1}{f-c})+N_{0}(r)+S(r,f),\notag\\
&2T(r,f)=N(r,\frac{1}{f-\Delta^{n}_{\eta}f})+S(r,f).
\end{align}
 That is
 \begin{align}
&T(r,f)+N_{0}(r)=m(r,\frac{1}{f-a})+m(r,\frac{1}{f-c})+S(r,f)\notag\\
&\leq m(r,\frac{\Delta^{n}_{\eta}f-\Delta^{n}_{\eta}a}{f-a})+m(r,\frac{\Delta^{n}_{\eta}f-\Delta^{n}_{\eta}c}{f-c})
+m(r,\frac{1}{\Delta^{n}_{\eta}f-\Delta^{n}_{\eta}a})+m(r,\frac{1}{\Delta^{n}_{\eta}f-\Delta^{n}_{\eta}c})\notag\\
&+S(r,f)\leq m(r,\frac{1}{\Delta^{n}_{\eta}f-\Delta^{n}_{\eta}a})+m(r,\frac{1}{\Delta^{n}_{\eta}f-\Delta^{n}_{\eta}c})+S(r,f).
\end{align}

We discuss two subcase.

{\bf Subcase 2.2.1} $\Delta^{n}_{\eta}a\not\equiv\Delta^{n}_{\eta}c$.

{\bf Subcase 2.2.1.1} $\Delta^{n}_{\eta}a\not\equiv a,c$ and $\Delta^{n}_{\eta}c\not\equiv a,c$. Set $F_{4}=\frac{1}{F_{2}}$ and $G_{4}=\frac{1}{G_{2}}$. With the same way to prove {\bf Subcase 2.1}, we only need to discuss $T(r,f)= N(r,\frac{1}{f-a})+S(r,f)$. It follows from (3.79) that
$$N(r,\frac{1}{f-c})+N_{0}(r)=S(r,f).$$
But in this case, we can also obtain a contradiction.

{\bf Subcase 2.2.2.2} $\Delta^{n}_{\eta}a\equiv a$ and $\Delta^{n}_{\eta}c\equiv c$. We set
\begin{align}
J&=(f-a)(\Delta^{n}_{\eta}a-\Delta^{n}_{\eta}c)-(\Delta^{n}_{\eta}f-\Delta^{n}_{\eta}a)(a-c)\notag\\
&=(f-c)(\Delta^{n}_{\eta}a-\Delta^{n}_{\eta}c)-(\Delta^{n}_{\eta}f-\Delta^{n}_{\eta}c)(a-c).
\end{align}
We claim that $J\not\equiv0$. Otherwise,  by the equalities $\Delta^{n}_{\eta}a\equiv a$, $\Delta^{n}_{\eta}c\equiv c$, and the definition of $J$, we can get $f\equiv \Delta^{n}_{\eta}f$, a contradiction. So $J\not\equiv0$. Then it is easy to see from (3.69) that
\begin{align}
 T(r,f)+N_{0}(r)&=m(r,\frac{1}{f-a})+m(r,\frac{1}{f-c})+S(r,f)\notag\\
 &=m(r,\frac{1}{f-a}+\frac{1}{f-c})+S(r,f)\notag\\
  &=m(r,\frac{J}{f-a}+\frac{J}{f-c})+m(r,\frac{1}{J})\notag\\
 & \leq T(r,f-\Delta^{n}_{\eta}f)-N(r,\frac{1}{f-\Delta^{n}_{\eta}f})+S(r,f)\notag\\
 &\leq T(r,f)-N(r,\frac{1}{f-\Delta^{n}_{\eta}f})+S(r,f)=S(r,f)
 \end{align}
which implies
\begin{align}
 T(r,f)=S(r,f),
\end{align}
a contradiction.

{\bf Subcase 2.2.2.3} $\Delta^{n}_{\eta}a\equiv c$ and $\Delta^{n}_{\eta}c\equiv a$.  We set
$$F_{4}=\frac{f-b}{f-a}\cdot\frac{c-a}{c-b}, G_{4}=\frac{\Delta^{n}_{\eta}f-b}{\Delta^{n}_{\eta}f-a}\cdot\frac{c-a}{c-b}.$$
Since $f\not\equiv \Delta^{n}_{\eta}f$, $F_{4}\not\equiv G_{4}$. And with a similar discussion of above, we can obtain $\Delta^{n}_{\eta}b\equiv c$ and $\Delta^{n}_{\eta}c\equiv b$. It follows from $\Delta^{n}_{\eta}c\equiv a$ that $a\equiv b$, a contradiction.

{\bf Case 2.2.2} $\Delta^{n}_{\eta}a\equiv\Delta^{n}_{\eta}c$. Then by (3.70) we have
\begin{align}
T(r,f)+N_{0}(r)&=m(r,\frac{1}{f-a})+m(r,\frac{1}{f-c})+S(r,f)\notag\\
&\leq m(r,\frac{1}{f-a}+\frac{1}{f-c})+S(r,f)\notag\\
&\leq m(r,\frac{\Delta^{n}_{\eta}f-\Delta^{n}_{\eta}a}{f-a}+\frac{\Delta^{n}_{\eta}f-\Delta^{n}_{\eta}c}{f-c})
+m(r,\frac{1}{\Delta^{n}_{\eta}f-\Delta^{n}_{\eta}a})\notag\\
&+S(r,f)\leq m(r,\frac{1}{\Delta^{n}_{\eta}f-\Delta^{n}_{\eta}a})+S(r,f)\notag\\
&\leq T(r,g)-N(r,\frac{1}{\Delta^{n}_{\eta}f-\Delta^{n}_{\eta}a})+S(r,f),
\end{align}
it deduces that
\begin{align}
N(r,\frac{1}{\Delta^{n}_{\eta}f-\Delta^{n}_{\eta}a})+N_{0}(r)=S(r,f).
\end{align}
If  $\Delta^{n}_{\eta}a\equiv a$, then with the same proof of  $\Delta_{\eta}b\equiv b$ in {\bf Subcase 2.1}, and by (3.69) and $T(r,f)= N(r,\frac{1}{f-b})+S(r,f)$, we can have
 \begin{align}
T(r,e^{h_{2}})\leq m(r,\frac{1}{f-c})+S(r,f)=S(r,f).
\end{align}
 And thus,
 \begin{align}
\overline{N}_{0}(r)+\overline{N}(r,\frac{1}{f-b})\leq m(r,\frac{1}{f-b})+S(r,f),
\end{align}
which follows from (3.77) that $N(r,\frac{1}{f-b})=S(r,f)$, and furthermore we get $T(r,f)=S(r,f)$, a contradiction. If $\Delta^{n}_{\eta}a=c$, then we can also obtain  a contradiction with a same method of proving {\bf Subcase 2.1}.

\

{\bf Conflict of Interest}  The author declares that there is  no conflict of interest regarding the publication of this paper.
\

{\bf Acknowledgements} The author would like to thank to anonymous referees for their helpful comments.


\end{document}